\documentclass[12pt,leqno,twoside]{amsart}
\usepackage{amssymb,amsmath,amsthm,soul,color}
\usepackage{t1enc}
\usepackage{a4,indentfirst,latexsym}
\usepackage{graphics}
\usepackage{mathrsfs}
\usepackage{cite,enumitem,graphicx}
\usepackage[colorlinks=true,urlcolor=blue,
citecolor=red,linkcolor=blue,linktocpage,pdfpagelabels,
bookmarksnumbered,bookmarksopen]{hyperref}
\usepackage[english]{babel}
\usepackage[left=2.61cm,right=2.61cm,top=2.72cm,bottom=2.72cm]{geometry}
\usepackage[normalem]{ulem}

\usepackage[hyperpageref]{backref} 

\newtheorem{Th}{Theorem}[section]
\newtheorem{Prop}[Th]{Proposition}
\newtheorem{Lem}[Th]{Lemma}

\newtheorem{Rem}[Th]{Remark}

\newenvironment{altproof}[1]
{\noindent
{\em Proof of {#1}}.}
{\nopagebreak\mbox{}\hfill $\Box$\par\addvspace{0.5cm}}

\newcommand{\wt}{\widetilde}

   \newcommand{\vp}{\varphi}
   
   \newcommand{\eps}{\varepsilon}

   \def\div{\mathop{\mathrm{div}}}

   \def\supp{\mathrm{supp}}

   \def\d{\diamond}

   \def\R{\mathbb{R}}

   \def\curl{\mathrm{curl}}

   \def\V{\mathcal{V}}

   \def\W{\mathcal{W}}

\def\rh{\rightharpoonup}
\def\irn{\int_{\rn}}
\def\rn{\mathbb{R}^N}

\newcommand{\cC}{{\mathcal C}}
\newcommand{\cD}{{\mathcal D}}

\newcommand{\cM}{{\mathcal M}}

\newcommand{\cV}{{\mathcal V}}
\newcommand{\cW}{{\mathcal W}}

\def\r{\mathbb{R}}
\def\r3{\mathbb{R}^3}

\def\rn{\mathbb{R}^N}

\def\eps{\varepsilon}
\def\rh{\rightharpoonup}
\def\irn{\int_{\rn}}

\def\io{\int_{\Omega}}

\def\ir3{\int_{\r3}}
\def\vp{\varphi}
\def\wt{\widetilde}
\def\wh{\widehat}
\def\ol{\overline}

\def\lam{\lambda}
\def\lam1{\lambda_1}

\def\d1p{\mathcal{D}^{1,p}}

\def\curlop{\nabla\times}
\newcommand{\weakto}{\rightharpoonup}

\numberwithin{equation}{section}

\title{On a critical curl-curl system of Lane--Emden type}

\author[I. Arkoudis]{Ioannis Arkoudis}
\author[J. Mederski]{Jaros\l aw Mederski}
\author[A. Szulkin]{Andrzej Szulkin}

\address[I. Arkoudis]{\newline\indent
	{Institute of Mathematics,
		\newline\indent 
		Polish Academy of Sciences,
		\newline\indent 
		ul. \'Sniadeckich 8, 00-656 Warsaw, Poland
	}
}
\email{\href{mailto:jarkoudis@impan.pl}{jarkoudis@impan.pl}}

\address[J. Mederski]{\newline\indent
	{Institute of Mathematics,
		\newline\indent 
		Polish Academy of Sciences,
		\newline\indent 
		ul. \'Sniadeckich 8, 00-656 Warsaw, Poland,		\newline\indent
		and
		\newline\indent 
		Faculty of Mathematics and Computer Science,		\newline\indent 
		Nicolaus Copernicus University, \newline\indent ul. Chopina 12/18,		 87-100 Toru\'n, Poland
	}
}
\email{\href{mailto:jmederski@impan.pl}{jmederski@impan.pl}}

\address[A. Szulkin]{\newline\indent 
	Department of Mathematics,
	\newline\indent 
	Stockholm University,
	\newline\indent 
	106 91 Stockholm, Sweden
}
\email{\href{mailto:andrzejs@math.su.se}{andrzejs@math.su.se}}


\hypersetup{hidelinks}
\begin{document}

\begin{abstract}
We study the system $\curlop\curlop u = |v|^{q-2}v$,
$\curlop\curlop v = |u|^{p-2}u$ of Lane--Emden type in
$\mathbb R^3$, with exponents $p,q>3$ on the critical hyperbola
$1/p+1/q=1/3$. We prove the existence of a ground state solution by reducing the
system to a quasilinear fourth-order curl-curl equation and minimizing the associated Sobolev-curl quotient. 
A concentration-compactness argument shows there are finitely many concentration
points for a Palais--Smale sequence $(u_n)$. This yields a.e.
 convergence of $\curlop\curlop u_n$ to $\curlop\curlop u$ where $u$ is the weak limit of $(u_n)$. We also prove that
the optimal Sobolev-curl constant is attained and is strictly larger
than the usual second-order Sobolev constant.
If $p=q$, we show $u$ and $v$ are necessarily equal for every solution. In an appendix we show that our concentration argument also applies to a scalar quasilinear fourth order equation. 
\end{abstract}

\maketitle

\section{Introduction} \label{sec:intro}

Lane--Emden type systems arise naturally in the study of nonlinear elliptic equations
and have been extensively investigated due to their strong connections with
astrophysics, geometry and variational methods.
A prototypical Lane--Emden system in $\mathbb{R}^N$ takes the form
\begin{equation}\label{LE-system}
\begin{cases}
-\Delta u = |v|^{q-2}v \\
-\Delta v = |u|^{p-2}u
\end{cases}
\qquad \text{in } \mathbb{R}^N,
\end{equation}
where $N \geq 3$ and the exponents $p,q>1$ satisfy suitable conditions.
In the critical case, i.e. when $(p,q)$ lies on the so--called critical hyperbola
\[
\frac{1}{p} + \frac{1}{q} = \frac{N-2}{N},
\]
system \eqref{LE-system} exhibits lack of compactness due to the invariance
under translations and dilations.

The Lane--Emden system \eqref{LE-system} is of Hamiltonian type and can be studied
by variational techniques. In the critical regime the existence and qualitative
properties of solutions are strongly influenced by concentration phenomena and
blow--up behavior. While positive solutions are by now well understood, the construction
of sign--changing or nodal solutions remains a challenging problem, especially in the
absence of additional symmetries or explicit formulas.

Motivated by these difficulties, several authors have developed refined variational
frameworks in order to overcome the lack of compactness.
A fruitful approach consists in exploiting symmetry properties of the problem or in
recasting the system into an equivalent higher--order equation via a reduction
procedure. These methods have led to the existence of nonradial and sign--changing
solutions for critical Lane--Emden systems and related equations.

Existence and uniqueness of a ground state solution to \eqref{LE-system} has been shown by Hulshof and van der Vorst \cite{Hulshof}. For a survey of equations of Lane--Emden type, both in bounded domains and in $\rn$, with extensive lists of references, we refer to Bonheure, Moreira dos Santos, Tavares \cite{bmt} and Pistoia, Schiera, Tavares \cite{pst}. 

In this paper we investigate an analogue of \eqref{LE-system} with $-\Delta$ replaced by the curl-curl operator $\curlop\curlop\cdot: \r3\to\r3$. To simplify notation, here and below we write $\curlop\curlop u$ for $\curlop(\curlop u)$ if this causes no ambiguity. More precisely, we study the critical curl-curl system 
\begin{equation} \label{system}
\begin{cases}
\curlop\curlop u = |v|^{q-2}v \\
\curlop\curlop v = |u|^{p-2}u \\
u\in \cD^{2,q'}(\curl;\r3) \\
v\in \cD^{2,p'}(\curl;\r3)
\end{cases}
\end{equation}
where 
\[
p,q>3, \quad p'=\frac{p}{p-1},  \quad q'=\frac{q}{q-1}
\]
 and
\begin{eqnarray}
\cD^{2,q'}(\curl;\r3) & := &  \{u\in L^p(\r3,\r3): \curlop\curlop u \in L^{q'}(\r3,\r3)\}, \label{Dcurl1} \\
\cD^{2,p'}(\curl;\r3) & := &  \{v\in L^q(\r3,\r3): \curlop\curlop v \in L^{p'}(\r3,\r3)\}. \label{Dcurl2}
\end{eqnarray}
These spaces can also be defined as completions of $\cC^{\infty}_0(\R^3,\R^3)$ with respect to the norms
 \[
 \|u\|_{q',p} := \big(|u|_p^2 + | \nabla \times \nabla \times u |_{q'}^2\big)^{1/2} \qquad \text{in }\cD^{2,q'}(\curl;\R^3)
 \] 
 and
 \[
 \|v\|_{p',q} :=\big(|v|_q^2 + | \nabla \times \nabla \times v |_{p'}^2\big)^{1/2} \qquad \text{in }\cD^{2,p'}(\curl;\R^3),
 \] 
see the arguments in the proof of \cite[Lemma A.1]{MedSz2}.
 Here and in the sequel $|\cdot|_r$ denotes the $L^r$-norm for $r\in[1,\infty]$. We shall assume that $p$ and $q$ lie on the critical hyperbola which for $N=3$ takes the form 
\begin{equation} \label{crithyp}
\dfrac{1}{p}+\dfrac{1}{q}=\dfrac{1}{3}.
\end{equation} 
Similarly as for \eqref{LE-system}, this hyperbola can be derived either via scaling or via the Sobolev embedding. Note that in view of \eqref{crithyp}, $p$ in \eqref{Dcurl1} is the critical Sobolev exponent, i.e. there is an embedding $\cD^{2,q'}(\r3,\r3) \hookrightarrow L^p(\r3,\r3)$, and similarly $q$ in \eqref{Dcurl2} is a critical Sobolev exponent. We shall write 
\[
p = (q')^{**} := \frac{3q'}{3-2q'}.
\]
 Here and in what follows  
\[
\cD^{2,q'}(\r3,\r3) := \{u\in L^p(\r3,\r3): \Delta u\in L^{q'}(\r3,\r3)\}
\]
endowed with the norm $|\Delta u|_{q'}$. This norm is well-known to be equivalent to the usual one involving all second derivatives, see e.g. \cite[Corollary 9.10]{GilbargTrudinger}.

We will study the system \eqref{system} by means of the equation 
\begin{equation}\label{equation}
\nabla \times \nabla \times \left(|\nabla \times \nabla \times u |^{q'-2}(\nabla \times \nabla \times u)\right)=|u|^{p-2}u \ , \quad u\in \cD^{2,q'}(\curl;\R^3).
\end{equation}
The following proposition justifies our approach.

\begin{Prop} \label{reduction}
	A pair of vector fields $(u,v)$ is a solution to the system \eqref{system} if and only if $u$ is a solution to the equation (\ref{equation}) and $v=|\nabla \times \nabla \times u|^{q'-2}(\nabla \times \nabla \times u)$.
\end{Prop}

This reduction is known for \eqref{LE-system}, see e.g. \cite{cs, Hulshof}, and the reduced equation is
\[
\Delta(|\Delta u|^{q'-2}\Delta u) = |u|^{p-2}u.
\]
Existence of a solution (a ground state) for this equation has been shown by P.L. Lions \cite{Lions} by using his concentration-compactness principle, as observed in \cite{Hulshof}.
The presence of the curl-curl operator introduces additional analytical difficulties
since the associated functional is strongly indefinite and the kernel of the operator
is infinite-dimensional. Therefore the arguments in \cite{Lions} need to be considerably modified.

The curl-curl problem for a single equation has been recently studied in several papers. See e.g. the surveys  \cite{BartschMederski3} and \cite{MederskiSchino} where also the physical background is discussed. Curl-curl systems of Lane--Emden type, in the subcritical case and in a bounded domain, have been considered in a very recent work by Chen, Song, Zhang \cite{csz1, csz2} and Song, Sun, Zou \cite{ssz}. In \cite{csz1, ssz} it is assumed that $p,q\in(2,6)$ while in \cite{csz2} $p\in(2,6)$ and $q=6$ (so $p,q$ are still subcritical in the sense that they satisfy $1/p+1/q>1/3$; however, also the case $p=q=6$ is treated there and it is shown that $u=v$ then, as in Theorem \ref{p=q} below). 

The functional corresponding to equation \eqref{equation} is $J: \cD^{2,q'}(\curl;\r3)\to \R$,
\[
J(u) := \frac1{q'}\ir3|\curlop\curlop u|^{q'}\,dx -\frac1p\ir3|u|^p\,dx
\]
where $\cD^{2,q'}(\curl;\r3)$ is as in \eqref{Dcurl1}. 

Our main result is

\begin{Th}\label{thm:main}
There is a ground state solution $u$ to \eqref{equation}, and hence via Proposition \ref{reduction}, $(u,v)$ with $v=|\nabla \times \nabla \times u|^{q'-2}(\nabla \times \nabla \times u)$ is a ground state solution to \eqref{system}.
\end{Th}

Recall that a ground state is a nontrivial solution with least energy among all nontrivial solutions.

It is clear that if $u$ is a ground state solution for \eqref{equation}, then $(u,v)$ with $v$ as in Proposition \ref{reduction} is a solution for \eqref{system}. To see that this is indeed a ground state we note that the functional corresponding to \eqref{system} is
\[
I(u,v) := \ir3\langle \curlop u, \curlop v\rangle\,dx -\frac1p\ir3|u|^p\,dx - \frac1q\ir3|v|^q\,dx.
\] 
Setting $v=|\nabla \times \nabla \times u|^{q'-2}(\nabla \times \nabla \times u)$ we obtain after integration by parts that $I(u,v) = J(u)$, hence the conclusion. 

We shall also show that

\begin{Th} \label{p=q}
If $p=q$, then any solution $(u,v)$ to \eqref{system} satisfies $u=v$.
\end{Th}

This result can be found in the paper \cite{ssz} where system \eqref{system} has been considered in a bounded domain and for $p,q\in (2,6)$ as we have already mentioned. The proof is exactly the same but we shall include it for the reader's convenience. We also note that if $p=q$, then $u=v$ solves the equation $\curlop\curlop u = |u|^4u$ which has been considered in \cite{MedSz}.

Recall that the Sobolev constant $S_{2,q'}$ in $\r3$ is defined by the inequality 
\begin{equation} \label{sobolev}
\ir3|\Delta u|^{q'}\,dx \ge S_{2,q'}\Big(\ir3|u|^p\,dx\Big)^{q'/p} 
\end{equation}
which should hold for all $u\in\cD^{2,q'}(\r3)$ and $S_{2,q'}$ is the largest constant with this property (as mentioned earlier, $p$ is the critical Sobolev exponent according to \eqref{crithyp} and $|\Delta u|_{q'}$ is the norm we use). Similarly as in \cite{MedSz, MedSz2} (see also \cite{FrankLoss2}) we define $S_{2,q',\curl}$ as the largest constant such that the inequality
\begin{equation} \label{Scurl}
\ir3|\curlop\curlop u|^{q'}\,dx \ge S_{2,q',\curl} \inf_{\substack{w\in \cD^{2,q'}(\curl;\r3) \\ \curlop w= 0}} \Big(\ir3|u+w|^p\,dx\Big)^{q'/p}
\end{equation}  
holds for all $u\in \cD^{2,q'}(\curl;\r3)$. 

In the proof of Theorem \ref{thm:main} we shall show that $S_{2,q',\curl}$ is attained by some $u\in \cD^{2,q'}(\curl;\r3)$ and this $u$ is a ground state solution. We also have the following relation:

\begin{Th}\label{thm:comparison}
There holds
\[
S_{2,q',\curl}>S_{2,q'}>0.
\]
Moreover, a ground state solution $u$ obtained in Theorem \ref{thm:main} is an optimizer in \eqref{Scurl}, i.e. equality holds there for this $u$.
\end{Th}

\medskip

The paper is organized as follows. Section \ref{sec:setting} contains preliminary material including the proof of Proposition \ref{reduction}. In Section \ref{sec:cca} we prove a crucial Proposition \ref{conccomp} whose consequence is a.e. convergence in \eqref{aeconv}. In Section \ref{sec:proofs} we prove Theorems \ref{thm:main}, \ref{p=q} and \ref{thm:comparison}. 
Appendix~\ref{app:scalar-compactness} applies our concentration-compactness argument to the scalar equation involving the $q'$-bilaplacian. This corrects a mistake in the proof of \cite[Proposition A.1]{cs} and gives an additional information that the $q'$-bilaplacian locally strongly converges in $\R^3\setminus\Sigma$ where $\Sigma$ is a (specified) finite set.

\section{Functional setting and preliminaries}\label{sec:setting}

Let $\cD^{1,r}(\R^3,\R^3)$ denote the completion of $\cC^{\infty}_0(\R^3,\R^3)$ with respect to the norm $|\nabla \cdot|_r$. In what follows $\lesssim$ and $\gtrsim$ denote the inequalities up to a multiplicative constant.
The following Helmholtz decomposition holds.

\begin{Lem}\label{defof} 
Let $1<r<\infty$ and 
\begin{eqnarray*}
\cV_r & := & \Big\{v\in L^r(\R^3,\R^3): \int_{\R^3}\langle v,\nabla \varphi\rangle\,dx=0
\text{ for every $\varphi\in \cC^\infty_0(\R^3)$} \Big\} \\
& = & \{v\in L^r(\R^3,\R^3): \div(v) = 0 \text{ in the sense of distributions}\}
\end{eqnarray*}
and
\begin{eqnarray*}
\cW_r & := & \Big\{w\in L^r(\R^3,\R^3): \int_{\R^3}\langle w,\curlop \varphi\rangle\,dx=0
\text{ for every $\varphi\in \cC^\infty_0(\R^3,\R^3)$} \Big\} \\
& = & \{ w\in L^r(\R^3,\R^3): \curlop w = 0 \text{ in the sense of distributions}\}.
\end{eqnarray*}
Then $\V_r$ and $\cW_r$ are closed subspaces of $L^r(\R^3,\R^3)$  and
	\begin{equation}\label{HelmholtzDec}
	L^r(\R^3,\R^3)=\V_r\oplus \W_r.
	\end{equation}
\end{Lem}

\begin{proof}

 Firstly, we easily check that $\V_r$ and $\cW_r$ are closed subspaces of $L^r(\R^3,\R^3)$.
	 
Now, take any $u\in L^r(\R^3,\R^3)$ and $u_n\in\cC_0^{\infty}(\R^3,\R^3)$ such that $u_n\to u$ in $L^r(\R^3,\R^3)$. Let $\vp_n\in\cC^{\infty}(\R^3)$ be the Newtonian potential of $\div(u_n)$, i.e. $\vp_n$ solves $\Delta \vp_n = \div(u_n)$.  Since $u_n\in\cC_0^{\infty}(\R^3,\R^3)$, then by \cite[Proposition 1]{Iwaniec}, $\nabla \vp_n\in L^r(\R^3,\R^3)$ and $|\nabla\vp_n|_r \le C_r|u_n|_r$ where $C_r>0$ is a constant. 
Hence
$$	v_n:=u_n-\nabla \vp_n\in  L^{r}(\R^3,\R^3).
	$$ 
Moreover,
$\Delta (\vp_n-\vp_m) = \div(u_n-u_m)$
and
$$|\nabla \vp_n-\nabla \vp_m|_r\leq C_r |u_n-u_m|_r.$$
Since $\div(v_n)=0$ pointwise and $(v_n)$ is a Cauchy sequence in $L^{r}(\R^3,\R^3)$,  we set $v:=\lim_{n\to\infty}v_n$ in $L^{r}(\R^3,\R^3)$ and get 
	$$\int_{\R^3}\langle v,\nabla \vp\rangle\,dx=\lim_{n\to\infty}\int_{\R^3}\langle v_n,\nabla \vp\rangle\,dx=0$$
	for any $\vp\in\cC_0^{\infty}(\R^3)$. Hence $\div(v) =0$ and $v\in\cV_r$.
	Note that
	$\nabla \vp_n=u_n-v_n\to u-v$ in $L^{r}(\R^3,\R^3)$. Since $\nabla\vp_n\in\cW_r$ and $\cW_r$ is closed,
	$u-v\in \cW_r$ and we get the decomposition
	$$u=v+(u-v)\in \cV_r+\cW_r.$$
	\indent Now take $v\in \cV_r\cap\W_r$. Since $\curlop v =0$,
	 by \cite[Lemma 1.1(i)]{Le}, $v=\nabla\xi$ for some $\xi\in W^{1,r}_{loc}(\R^3)$. Since $\div(v) = 0$, $\xi$ is harmonic and so is $v$. Since $v\in L^{r}(\R^3,\R^3)$, by the mean-value formula we infer that $v=0$, so \eqref{HelmholtzDec} holds.
\end{proof}

\begin{altproof}{Proposition \ref{reduction}}
Let $u\in \cD^{2,q'}(\curl;\R^3)$ be a solution of the equation (\ref{equation}) and  set 
\begin{equation} \label{def_v}
v:=|\nabla \times \nabla \times u|^{q'-2}(\nabla \times \nabla \times u).
\end{equation}
Since $\curlop\curlop u \in L^{q'}(\r3,\r3)$, $v\in  L^{q}(\r3,\r3)$. Equation
\eqref{equation} gives $\curlop\curlop v=|u|^{p-2}u\in L^{p'}(\r3,\r3)$. So $v\in\cD^{2,p'}(\curl;\r3)$
and by \eqref{def_v}, $\curlop\curlop u = |v|^{q-2}v$. Hence  $(u,v)$ solves \eqref{system}.

 To prove the converse statement, suppose  $(u,v)$ is a solution to \eqref{system}. Then $\curlop \curlop u = |v|^{q-2}v \ \in L^{q'}(\R^3,\R^3)$,  $\curlop \curlop v = |u|^{p-2}u \ \in L^{p'}(\R^3,\R^3)$ and the system is satisfied almost everywhere. Substituting $v = |\curlop \curlop u|^{q'-2} \curlop \curlop  u$ into the second equation leads to
\begin{equation*}
\curlop \curlop (|\curlop \curlop u|^{q'-2} \curlop \curlop u)=|u|^{p-2}u, \quad \hbox{a.e. in }\R^3,
\end{equation*}
and the conclusion follows since $u\in \cD^{2,q'}(\curl;\r3)$.
\end{altproof}

Let $u\in\cD^{2,q'}(\curl,\r3)$. In the Helmholtz decomposition we shall write $u=v+w$ where $v\in\cV$, $w\in\cW$ and $\cV,\cW$ are as in the next lemma. So here and below $v$ is \textbf{not the same as in \eqref{system}}. This will cause no confusion as in what follows we shall only work with  equation \eqref{equation} and the functional $J$.

\begin{Lem}\label{Hdecomp} 
Let $\cV := \cV_p\cap \cD^{2,q'}(\curl;\r3)$ and $\cW := \cW_p\cap \cD^{2,q'}(\curl;\r3)$ (recall $p = (q')^{**}$). Then
$\cV$ and $\cW$ are closed subspaces of $\cD^{2,q'}(\curl;\r3)$  and
\begin{equation}\label{Helmholtz2}
\cD^{2,q'}(\curl;\r3)=\cV\oplus \cW.
\end{equation}
Moreover, $\cV\subset\cD^{2,q'}(\r3,\r3)$ and $|\Delta v|_{q'} = |\curlop\curlop v|_{q'}$ for all $v\in\cV$. 
\end{Lem}

\begin{proof}
That $\cV$ and $\cW$ are closed and \eqref{Helmholtz2} holds is an immediate consequence of Lemma \ref{defof} and the obvious identity
\[
(\cV_p\oplus \cW_p)\cap \cD^{2,q'}(\curl;\r3) = (\cV_p\cap \cD^{2,q'}(\curl;\r3))\oplus (\cW_p\cap \cD^{2,q'}(\curl;\r3)).
\]
If $v\in\cV$, then $\curlop\curlop v = -\Delta v\in L^{q'}(\r3,\r3)$. Hence $|\Delta v|_{q'} = |\curlop\curlop v|_{q'}$ and $\cV\subset\cD^{2,q'}(\r3,\r3)$.
\end{proof}

\begin{proof}[Beginning of the proof of Theorem \ref{thm:comparison}]
We shall show that 
\begin{equation} \label{sobolev2}
 S_{2,q',\curl} \ge S_{2,q'}. 
\end{equation}

It is well known that the Sobolev inequality \eqref{sobolev} remains valid, with the same
constant, for vector-valued functions. However, since we could not find a convenient reference, we provide a proof.
Let 
\begin{equation} \label{newtonian}
I_2(f)(x):=\frac1{4\pi}\int_{\r3}\frac{f(y)}{|x-y|}\,dy = (-\Delta)^{-1}(f) 
\end{equation} 
be the Newtonian potential of $f$ (i.e. the Riesz potential $I_\alpha(f)$ for $\alpha=2$, cf. \cite[Sec. V.1, eqs. (3)-(4)]{Stein}). 
Let
$v\in\cD^{2,q'}(\r3,\r3)$ and set  $f:=|\Delta v|$. Then $f\in L^{q'}(\r3)$ and the Hardy--Littlewood--Sobolev inequality \cite[Theorem V.1]{Stein}  implies
\(I_2(f)\in L^p(\r3)\). Since \(-\Delta\,I_2(f)=f\), it follows that
\(I_2(f)\in \cD^{2,q'}(\r3)\) and the  scalar Sobolev
inequality applies to it. Moreover, $v=I_2(-\Delta v)$.  
So by \eqref{newtonian}, 
\[
|v| \leq I_2(f)\quad\text{a.e. in }\r3.
\]
Applying the scalar Sobolev inequality \eqref{sobolev} to $I_2(f)$, we obtain 
\begin{eqnarray} \label{sobolev3}
\int_{\r3}|\Delta v|^{q'}\,dx & = & \int_{\r3}|f|^{q'}\,dx = \ir3|\Delta I_2(f)|^{q'}\,dx \\
&  \ge & S_{2,q'} \left(\int_{\r3}|I_2(f)|^p\,dx\right)^{q'/p}
\geq
S_{2,q'}\left(\int_{\r3}|v|^p\,dx\right)^{q'/p}. \nonumber
\end{eqnarray}
On the other hand, the Sobolev constant for vector-valued functions cannot be larger because the inequality must be satisfied for functions $v=(v_1,0,0)$. This proves the validity of \eqref{sobolev} for vector-valued functions.

 Now we proceed to \eqref{sobolev2}. Fix \(u\in\cD^{2,q'}(\curl;\r3)\). By Lemma \ref{Hdecomp}, write
\[
u=v+w_0,
\qquad v\in\cV,\quad w_0\in\cW.
\] 
Then 
\[
\inf_{\substack{w\in\cD^{2,q'}(\curl;\r3)\\ \curlop w=0}}
\left(\int_{\r3}|u+w|^p\,dx\right)^{q'/p}
\leq
\left(\int_{\r3}|u-w_0|^p\,dx\right)^{q'/p}
=
\left(\int_{\r3}|v|^p\,dx\right)^{q'/p}.
\]
Since $\curlop\curlop u=-\Delta v$,  combining \eqref{sobolev3} with the  inequality above shows that \eqref{Scurl} holds with
\(S_{2,q'}\) in place of \(S_{2,q',\curl}\). As
\(S_{2,q',\curl}\) is the largest admissible constant, it follows that $S_{2,q',\curl} \ge S_{2,q'}$. This fact will be needed in the proof of Theorem \ref{thm:main}. Subsequently, Theorem \ref{thm:main} will be used in order to complete the proof of Theorem \ref{thm:comparison} by showing that the inequality above is strict.
\end{proof}

\section{Concentration-compactness analysis} \label{sec:cca}

We denote the space of finite measures by $\cM(\R^3)$ and set
$$\|\mu\| := \sup_{\substack{u \in \cC_0(\mathbb{R}^3) \\ |u|_\infty = 1}} |\langle \mu, u \rangle|\quad\hbox{ for }\mu\in \cM(\R^3).$$
Recall that if $\mu\geq 0$, then $\|\mu\|=\langle \mu, 1 \rangle$, see e.g. \cite[Theorem 1.39]{Willem}. 

Let $u\in\cD^{2,q'}(\curl;\r3)$. By convexity and coercivity there exists a unique $w(u)$ such that
\begin{equation} \label{minimizer}
\ir3|u+w(u)|^p\,dx = \inf_{w\in\cW}\ir3|u+w|^p\,dx. 
\end{equation}
This implies 
\begin{equation} \label{eq:eqf}
\int_{\R^3}\langle |u+w|^{p-2}(u+w),\zeta\rangle\,dx  =0 \text{ for all } \zeta\in\cW\quad\text{if and only if }  w=w(u).
\end{equation}
Moreover, the mapping $u\mapsto w(u)$ is continuous in the $L^p$-topology, and hence also in the $\cD^{2,q'}(\curl;\r3)$-topology, see \cite[Lemma 2.3]{MedSz2} (the exponent $p$ here corresponds to $p^*$ in \cite{MedSz2}).

By Lemma \ref{Hdecomp}, $u+w(u)=v+w(v)\in \cV\oplus\cW$ for some $v\in\cV$. This implies that $S_{2,q',\curl}$ can be characterized as follows:
\begin{equation} \label{charact}
S_{2,q',\curl} = \inf_{\substack{u\in \cD^{2,q'}(\curl;\r3) \\ \curlop\curlop u\neq 0}} \frac{|\curlop\curlop u|_{q'}^{q'}}{|u+w(u)|_{p}^{q'}} = \inf_{v\in \cV\setminus\{0\}} \frac{|\curlop\curlop v|_{q'}^{q'}}{|v+w(v)|_{p}^{q'}}. 
\end{equation}
We set
$$\cM:=\big\{v+w(v):\; v\in\cV\big\}$$
and below we show that $\cM$ is compactly embedded in $L^r_{loc}(\R^3,\R^3)$ for any $1\leq  r<p$.
\begin{Lem} \label{Lem:WConcentration}
Suppose $v_n+w(v_n)\in\cM$,  $v_n\weakto v_0$ in $\cV$, $v_n\to v_0$ a.e. in $\R^3$, $|\Delta(v_n-v_0)|^{q'}\weakto \mu$ and $|v_n-v_0|^{p}\rh \rho$ in $\cM(\R^3)$. Then there exist an at most countable set $\{x_j:j\in I\}\subset\R^3$ and positive weights $\{\mu_j\}_{j\in I}$, $\{\rho_j\}_{j\in I}$ such that
\begin{equation} \label{measures}
\mu\geq\sum_{j\in I}\mu_j\delta_{x_j}, \quad \rho = \sum_{j\in I}\rho_j\delta_{x_j},
\end{equation}
and passing to a subsequence,
$w(v_n)\rh w(v_0)$ in $\cW$, $w(v_n)\to w(v_0)$ a.e. in $\R^3$ and in $L^r_{loc}(\R^3,\R^3)$ for any $1\leq  r<p$. In particular, $\cM$ is compactly embedded  in $L^r_{loc}(\R^3,\R^3)$ for any $1\leq  r<p$.
\end{Lem}

\begin{Rem} \label{remark}
\emph{
Note that if $v_n\rh v_0$ in $\cV$, then the assumptions of Lemma \ref{Lem:WConcentration} are necessarily satisfied after passing to a subsequence.
}
\end{Rem}

\begin{proof}[Sketch of proof of Lemma \ref{Lem:WConcentration}]
	The argument is a second-order counterpart of
	\cite[Theorem 2.4 and Appendix~B]{MedSz2}.  
	Existence of the set $I$, the weights $\mu_{x_j}, \rho_{x_j}$ and the relations \eqref{measures} follow directly from \cite[Concentration-Compactness Lemma II]{Struwe}.

	Since 
	$$\int_{\R^3}|v_n+w(v_n)|^{p}\,dx \le \int_{\R^3}|v_n|^{p},$$
	we see that $(w(v_n))$ is bounded.
	Thus, taking a subsequence we may assume $w(v_n)\weakto w_0$ in $\cW$ for some $w_0\in\cW$. The rest of the proof is exactly the same as in \cite[Appendix B]{MedSz2}: we show taking subsequences that $w(v_n)\to w_0$ a.e., then that $w(v_n)\to w_0$ in $L^r_{loc}(\r3,\r3)$ for $1\le r<p$, and finally, that $w_0=w(v_0)$. (We also note that as $p>3$, the case (B.2) and not (B.3) in \cite[Appendix B]{MedSz} is relevant here.)
	
\end{proof}

\begin{Prop}
\label{conccomp}
Let \((u_n)\subset\cD^{2,q'}(\curl;\r3)\) be a bounded Palais--Smale sequence for \(J\) such that
\[
u_n=v_n+w(v_n)\in\cV\oplus\cW
\]
(i.e. $u_n\in\cM$). Suppose that
\[
u_n\rightharpoonup u=v+w(v), \quad
u_n\to u \quad\text{a.e. in }\r3
\]
and
\[
|\curlop\curlop(u_n-u)|^{q'} = |\Delta(v_n-v)|^{q'}\rightharpoonup\mu,
\quad
|v_n-v|^p \rightharpoonup\rho,
\quad
|u_n-u|^p \rightharpoonup\nu
\]
in \(\cM(\r3)\).
Then 
there exist constants \(c,C,\wt C>0\), independent of \(j\), such that
\[
c\,\rho_j^{q'/p}\leq\mu_j\leq C\nu_j^{1/p'}\rho_j^{1/p}
\]
and
\begin{equation}\label{eq:nurho}
\nu_j\leq \wt C\rho_j
\end{equation}
where $\rho_j,\mu_j$ are as in Lemma \ref{Lem:WConcentration} and $\nu_j$ are the corresponding weights for $\nu$. 
In particular, there exists \(\rho_*>0\), independent of \(j\), such that
\[
\rho_j\geq\rho_*
\qquad\text{for every }j\in I.
\]
Hence the concentration set
\[
\Sigma:=\{x_j:j\in I\}
\]
is finite.
\end{Prop}

\begin{proof}
Note that by Lemma \ref{Lem:WConcentration} $u\in \cM$, so indeed $u = v+w(v)$ for some $v\in\cV$. Set
\[
z_n:=v_n-v,
\qquad
q_n:=w(v_n)-w(v).
\]
Clearly
\(z_n\rightharpoonup0\) in \(\cD^{2,q'}(\r3,\r3)\)
and
\[
\curlop\curlop(u_n-u)
=\curlop\curlop z_n
=-\Delta z_n.
\]
Since $(z_n)$ is bounded in
$\cD^{2,q'}(\r3,\r3)$ and
$\curlop\curlop z_n=-\Delta z_n$, if
\(K\Subset K'\Subset\r3\), then the local Calder\'on--Zygmund estimate
(see \cite[Theorem 3]{DiFrattaFiorenza} and
\cite[Theorem 9.11]{GilbargTrudinger}) gives
\[
\|z_n\|_{W^{2,q'}(K,\r3)}
\leq C_{K,K'}
\left(
|\Delta z_n|_{L^{q'}(K',\r3)}+
|z_n|_{L^{q'}(K',\r3)}
\right).
\]
Since $p>q'$, the right-hand side is bounded. The Rellich-Kondrachov compactness theorem and
 weak convergence of \(z_n\) therefore give, after passing to a
subsequence,
\begin{equation}\label{eq:localzn}
z_n\to0
\quad\text{in }W^{1,s}_{loc}(\r3,\r3)
\quad\text{for }1\leq s<\frac{3q'}{3-q'}
\end{equation}
and hence a.e. in $\r3$ and in $L^r_{loc}(\r3,\r3)$ for every
$1\leq r<p$.

For every fixed $\vp\in\cC_0^\infty(\r3)$, the second-order
Sobolev inequality gives
\[
S_{2,q'}
\left(
\int_{\r3}|\vp z_n|^p\,dx
\right)^{q'/p}
\leq
\int_{\r3}|\Delta(\vp z_n)|^{q'}\,dx.
\]
Writing
\[
\Delta(\vp z_n)=\vp\Delta z_n+r_n
\quad \text{where} \quad
r_n:=2\nabla\vp\cdot\nabla z_n+(\Delta\vp)z_n,
\]
we have $r_n\to0$ in $L^{q'}(\r3,\r3)$. Hence
\[
\int_{\r3}|\Delta(\vp z_n)|^{q'}\,dx
=
\int_{\r3}|\vp|^{q'}|\Delta z_n|^{q'}\,dx+o(1),
\]
and consequently
\[
S_{2,q'}
\left(
\int_{\r3}|\vp|^p\,d\rho
\right)^{q'/p}
\leq
\int_{\r3}|\vp|^{q'}\,d\mu.
\]
By \cite[Concentration--Compactness Lemma II]{Struwe},
\begin{equation}\label{eq:Sestrhomi}
S_{2,q'}\rho_j^{q'/p}\leq\mu_j
\qquad\text{for every }j\in I.
\end{equation}
We next compare the \(L^p\)-concentration of \(u_n-u\) and \(z_n\). Let
$\Omega$ be a bounded domain and let $\vp\in \cC_0^\infty(\Omega )$ be a nonnegative function. Fix $r\in(3,p)$. Since \(q_n\in \cW\), by \cite[Lemma 1.1]{Le} there exists
\[
\xi_n\in W^{1,p}_{loc}(\r3)
\]
such that
$q_n=\nabla\xi_n$,
and we may assume without loss of generality that $\io\xi_n\,dx = 0$. By the Poincar\'e inequality and Lemma \ref{Lem:WConcentration},
\begin{equation} \label{poincare}
\|\xi_n\|_{W^{1,r}(\Omega)} \le C_\Omega|\nabla\xi_n|_{L^r(\Omega,\r3)} = C_\Omega|q_n|_{L^r(\Omega,\r3)} \to 0.
\end{equation}
By Morrey's embedding theorem we get $\xi_n\to 0$ in $L^\infty(\Omega)$.

 Let $\psi_n := \nabla(\vp\xi_n) = \vp q_n + \xi_n\nabla\vp$. Then $\psi_n\in \cW$ and $(\psi_n)$ is uniformly bounded in $\cW$. As $u_n,u\in\cM$, we have 
\begin{equation} \label{eqq}
\ir3\langle |u_n|^{p-2}u_n-|u|^{p-2}u, \psi_n\rangle\,dx = 0
\end{equation}
according to \eqref{eq:eqf}.
Since $z_n\rh 0$ in $L^p(\r3,\r3)$,
\begin{equation} \label{eqq2}
\ir3\vp\langle |u|^{p-2}u, z_n\rangle\,dx \to 0.
\end{equation}
Using
\[
\vp(u_n-u)
=\vp z_n+\psi_n-\xi_n\nabla\vp,
\]
\eqref{eqq}, \eqref{eqq2} and H\"older's inequality, we obtain
\[
\begin{aligned}
&\int_{\r3}\vp\left\langle
|u_n|^{p-2}u_n-|u|^{p-2}u,u_n-u \right\rangle dx = \ir3 \left\langle
|u_n|^{p-2}u_n-|u|^{p-2}u, \vp z_n + \psi_n - \xi_n\nabla\vp \right\rangle dx
\\
&\qquad =  \ir3 \left\langle \vp |u_n|^{p-2}u_n, z_n \right\rangle dx + o(1) \leq
\left(\int_{\r3}\vp|u_n|^p\,dx\right)^{(p-1)/p}
\left(\int_{\r3}\vp|z_n|^p\,dx\right)^{1/p}
+o(1).
\end{aligned}
\]
In view of the Brezis-Lieb lemma we obtain
\begin{eqnarray*}
	&&\ir3\vp\langle |u_n|^{p-2}u_n - |u|^{p-2}u, u_n-u\rangle\,dx \\
	& \le & \Big(\ir3\vp |u_n-u|^{p}\,dx+\ir3\vp |u|^{p}\,dx\Big)^{(p-1)/p}\Big(\ir3\vp|z_n|^p\,dx\Big)^{1/p} + o(1).
\end{eqnarray*}
Since
\[
C_1|u_n-u|^p \le \langle |u_n|^{p-2}u_n - |u|^{p-2}u, u_n-u\rangle
\]
for some $C_1>0$ (see  e.g. \cite{Simon}), we have 
\[
C_1\ir3\vp|u_n-u|^p\,dx \le \Big(\ir3\vp|u_n-u|^p\,dx+\ir3\vp |u|^{p}\,dx\Big)^{(p-1)/p}\Big(\ir3\vp|z_n|^p\,dx\Big)^{1/p} + o(1).
\]
Passing to the limit, we obtain 
\begin{equation} \label{eqeq}
C_1\ir3\vp\,d\nu\le \Big(\ir3\vp\,d\nu+\ir3\vp |u|^{p}\,dx\Big)^{(p-1)/p}\Big(\ir3\vp\,d\rho\Big)^{1/p}.
\end{equation}
Fix \(j\in I\) and let \(\vp_{j,\varepsilon}\in\cC_0^\infty(\r3,[0,1])\) satisfy
\begin{equation} \label{phije}
\vp_{j,\varepsilon}=1\ \text{ on }B_\varepsilon(x_j),
\qquad
\operatorname{supp}(\vp_{j,\varepsilon})
\subset B_{2\varepsilon}(x_j).
\end{equation}
Applying the preceding inequality with \(\vp=\vp_{j,\varepsilon}\)
and then letting \(\varepsilon\to0\),  we obtain  $0\le\nu_j\le\wt C\rho_j$ for some $\wt C>0$ and all $j$, hence \eqref{eq:nurho} holds.

For \(s>1\), put \(A_s(a):=|a|^{s-2}a\), and set 
\[
f_n:=A_{q'}(\curlop\curlop u_n) = A_{q'}(\curlop\curlop v_n).
\]
The sequence \((f_n)\) is bounded in \(L^q(\r3,\r3)\). 
For every fixed $\vp\in\cC_0^\infty(\r3,[0,1])$, in view of
\eqref{eq:localzn},
\begin{equation}\label{eq:etaxi}
\begin{aligned}
R_n
&:=\curlop\curlop(\vp z_n)
-\vp\curlop\curlop z_n
\\
&=\curlop(\nabla\vp\times z_n)
+\nabla\vp\times\curlop z_n
\to 0
\quad\text{in }L^{q'}(\r3,\r3).
\end{aligned}
\end{equation}
The sequence \((\vp z_n)\) is bounded in
\(\cD^{2,q'}(\curl;\r3)\). Since $(u_n)$ is a Palais--Smale sequence, $J'(u_n)(\vp z_n)\to 0$  and by
\eqref{eq:etaxi} we obtain 
\begin{equation}\label{eq:etaxi2}
\int_{\r3}\vp
\left\langle f_n,\curlop\curlop z_n\right\rangle dx
=
\int_{\r3}\vp
\left\langle |u_n|^{p-2}u_n,z_n\right\rangle dx
+o(1).
\end{equation} 
The following elementary estimate holds:
\begin{equation} \label{elementary}
|a|^{q'}\lesssim
\left\langle A_{q'}(a+b),a\right\rangle
+|b|^{q'},
\qquad a,b\in\r3.
\end{equation}
Indeed,
\[
|a+b|^{q'}
=
\left\langle A_{q'}(a+b),a\right\rangle
+
\left\langle A_{q'}(a+b),b\right\rangle,
\]
and Young's inequality gives
\[
\left|
\left\langle A_{q'}(a+b),b\right\rangle
\right|
\leq
\frac12|a+b|^{q'}
+
C |b|^{q'}
\]
for a constant $C>0$. Now the inequality $|a|^{q'}\leq2^{q'-1}(|a+b|^{q'}+|b|^{q'})$
yields the claim.
We apply this inequality with
\[
a=\curlop\curlop z_n,
\qquad
b=\curlop\curlop v
\] 
and use  \eqref{eq:etaxi2} to obtain
\[
\int_{\r3}\vp 
|\curlop\curlop z_n|^{q'}\,dx
\lesssim
\int_{\r3} \vp|u_n|^{p-1}|z_n|\,dx+\int_{\r3} \vp|\curlop\curlop v|^{q'}\,dx+o(1).
\]
Since \(u_n=(u_n-u)+u\), the triangle and the H\"older inequalities give
\[
\begin{aligned}
\int_{\r3}\vp
|u_n|^{p-1}|z_n|\,dx
&\lesssim\left(\int_{\r3}\vp|u_n-u|^p\,dx\right)^{1/p'}\left(\int_{\r3}\vp|z_n|^p\,dx\right)^{1/p}
\\
&\quad+\left(
\int_{\r3}\vp|u|^p\,dx
\right)^{1/p'}
\left(
\int_{\r3}\vp|z_n|^p\,dx
\right)^{1/p}.
\end{aligned}
\]
Fix \(j\in I\) and let $\vp_{j,\eps}$ be as in \eqref{phije}. For each fixed \(\varepsilon>0\) we apply the preceding two inequalities with
\(\vp=\vp_{j,\varepsilon}\) and  let \(n\to\infty\). We obtain
\[
\begin{aligned}
\int_{\r3}\vp_{j,\varepsilon}\,d\mu
&\lesssim
\left(
\int_{\r3}\vp_{j,\varepsilon}\,d\nu\right)^{1/p'}
\left(
\int_{\r3}\vp_{j,\varepsilon}\,d\rho
\right)^{1/p}
\\
&\quad+\left(
\int_{\r3}\vp_{j,\varepsilon}|u|^p\,dx
\right)^{1/p'}
\left(
\int_{\r3}\vp_{j,\varepsilon}\,d\rho
\right)^{1/p}
\\
&\quad+\int_{\r3}\vp_{j,\varepsilon}
|\curlop\curlop v|^{q'}\,dx.
\end{aligned}
\]
Now let \(\varepsilon\to0\). Then
\[
\int_{\r3}\vp_{j,\varepsilon}\,d\mu\to\mu_j,
\qquad
\int_{\r3}\vp_{j,\varepsilon}\,d\nu\to\nu_j,
\qquad
\int_{\r3}\vp_{j,\varepsilon}\,d\rho\to\rho_j
\]
and
\[
\int_{\r3}\vp_{j,\varepsilon}|u|^p\,dx\to0,
\qquad
\int_{\r3}\vp_{j,\varepsilon}
|\curlop\curlop v|^{q'}\,dx\to0.
\]
Hence
\[
\mu_j\leq C\nu_j^{1/p'}\rho_j^{1/p}
\]
for some $C>0$. Finally, combining the measure estimates with \eqref{eq:Sestrhomi} and \eqref{eq:nurho}, we find that
\[
S_{2,q'}\rho_j^{q'/p}
\le\mu_j
\le C\nu_j^{1/p'}\rho_j^{1/p}
\le D\rho_j.
\]
Recalling that $\rho_j>0$, we have
\[
\rho_j
\ge
\left(\frac{S_{2,q'}}{D}\right)^{\frac{p}{p-q'}}
=:\rho_*>0.
\]
Since
\(\|\rho\| = \sum_{j\in I}\rho_j < \infty\),
the set $I$, and hence \(\Sigma=\{x_j:j\in I\}\), must be finite.
\end{proof}

\begin{Rem}
\emph{
For every nonnegative $\vp\in\cC_0^\infty(\r3\setminus\Sigma)$,
we have $\ir3\vp\,d\rho=0$. As
\eqref{eqeq}  gives $\ir3\vp\,d\nu=0$,
 $\nu$ is supported on $\Sigma$ and
$\nu=\sum_{j\in I}\nu_j\delta_{x_j}$. }
\end{Rem}

\section{Proofs of Theorems \ref{thm:main}-\ref{thm:comparison}} \label{sec:proofs}

\begin{altproof}{Theorem \ref{thm:main}}
For \(v\in\cV\), let \(m(v):=v+w(v)\). By \eqref{charact},
\begin{equation}\label{eq:S-reduced}
S_{2,q',\curl}
=
\inf_{v\in\cV\setminus\{0\}}
\frac{\displaystyle\int_{\r3}|\curlop\curlop v|^{q'}\,dx}
{\displaystyle\left(\int_{\r3}|m(v)|^p\,dx\right)^{q'/p}}.
\end{equation}
By the estimate proved at the end of Section 2,
\(S_{2,q',\curl}\geq S_{2,q'}>0\).
We first construct a constrained Palais--Smale sequence. As we have already noted, the map
\(v\mapsto m(v)\) is continuous. 
Using the argument of the proof of property (ii) on p. 4320 in \cite{BartschMederski2} we see that the map $\Psi: v\mapsto\int_{\r3}|m(v)|^p\,dx$ is of class $\cC^1$, with the derivative given by 
\[
\Psi'(v)(\phi)
=p\int_{\r3}
\langle |m(v)|^{p-2}m(v),\phi\rangle\,dx,
\qquad \phi\in\cV.
\]
 Uniqueness of the minimizer in \eqref{minimizer} also gives
\(m(tv)=tm(v)\) for \(t>0\), see \cite[Lemma 4.2]{MedSz}. 
Since $\Psi'(v)(v)=p\Psi(v)>0$ for $v\ne0$,
\[
\mathcal C:=\left\{v\in\cV:
\int_{\r3}|m(v)|^p\,dx=1\right\}
\]
is a nonempty closed, and hence complete,
\(\cC^1\)-submanifold of \(\cV\) and we may assume that $v$ in \eqref{eq:S-reduced} is in $\cC$. Applying Ekeland's variational
principle to
\[
v\longmapsto\int_{\r3}|\curlop\curlop v|^{q'}\,dx
\]
on the constraint \(\mathcal C\), we obtain a sequence \((v_n)\subset\mathcal C\) and numbers \(c_n\in\R\),
\(\varepsilon_n\to0\)
such that
\[
\int_{\r3}|\curlop\curlop v_n|^{q'}\,dx\longrightarrow S_{2,q',\curl}
\]
and
\begin{equation}\label{eq:constrained-PS}
\begin{aligned}
\bigg|&\int_{\r3}
\big\langle
|\curlop\curlop v_n|^{q'-2}\curlop\curlop v_n,
\curlop\curlop\phi
\big\rangle\,dx\\
&-c_n\int_{\r3}
\langle |m(v_n)|^{p-2}m(v_n),\phi\rangle\,dx\bigg|
\leq\varepsilon_n\|\phi\|_{\cV},
\end{aligned}
\end{equation}
for every \(\phi\in\cV\) \cite[Theorem 3.1 and Corollary 3.4]{Ekeland}, see also \cite[Theorem 8.5]{Willem}.  By \eqref{eq:eqf},
\[
\int_{\r3}
\langle |m(v_n)|^{p-2}m(v_n),v_n\rangle\,dx = \int_{\r3}
\langle |m(v_n)|^{p-2}m(v_n),m(v_n)\rangle\,dx
=\int_{\r3}|m(v_n)|^p\,dx=1.
\]
Therefore testing \eqref{eq:constrained-PS} with \(v_n\) yields
\[
c_n=\int_{\r3}|\curlop\curlop v_n|^{q'}\,dx+o(1)
\to S_{2,q',\curl}.
\]
For $c_n>0$, put
\[
t_n:=c_n^{1/(p-q')},\qquad
\widehat v_n:=t_nv_n,\qquad
\widehat u_n:=m(\widehat v_n)=t_nm(v_n).
\]
The last identity follows from the homogeneity of \(m\) noted above.
For \(\phi\in\cV\), \eqref{eq:constrained-PS} gives
\[
\big|J'(\widehat u_n)(\phi)\big|
\leq t_n^{q'-1}\varepsilon_n\|\phi\|_{\cV} 
\]
whereas \(J'(\widehat u_n)(\psi)=0\) for every \(\psi\in\cW\) by
\eqref{eq:eqf}. Hence
\[
J'(\widehat u_n)\to 0
\quad\text{in }\big(\cD^{2,q'}(\curl;\r3)\big)^*.
\]
Furthermore,
\begin{equation}\label{eq:PS-masses}
\int_{\r3}|\curlop\curlop\widehat u_n|^{q'}\,dx\to S_{2,q',\curl}^{p/(p-q')}\quad\hbox{ and }\quad
\int_{\r3}|\widehat u_n|^p\,dx\to S_{2,q',\curl}^{p/(p-q')}.
\end{equation}
Thus
\begin{equation}\label{eq:ground-level}
J(\widehat u_n)\longrightarrow
c:=
\left(\frac1{q'}-\frac1p\right)S_{2,q',\curl}^{p/(p-q')}=\frac23S_{2,q',\curl}^{p/(p-q')}.
\end{equation}
Next we show that a subsequence $(\wh u_n)$, after a suitable translation and dilation, goes weakly to some $\wh u\ne 0$. 
Since \(\widehat v_n\in\cV\), Lemma \ref{Hdecomp} and the Sobolev inequality give
\[
|\nabla\widehat v_n|_{\frac{3q'}{3-q'}}
\leq C|\Delta\widehat v_n|_{q'}
=C|\curlop\curlop\widehat v_n|_{q'}.
\]
Thus \((\widehat v_n)\) is bounded in
\(\cD^{1,\frac{3q'}{3-q'}}(\r3,\r3)\). Notice that
\[
1<\frac{3q'}{3-q'}<3
\quad \text{and} \quad
\left(\frac{3q'}{3-q'}\right)^*
=\frac{3q'}{3-2q'}=p.
\]
Moreover, the minimizing property of \(w(\widehat v_n)\) and
\eqref{eq:PS-masses} imply
\[
|\widehat v_n|_p^p
\geq|m(\widehat v_n)|_p^p
=|\widehat u_n|_p^p\to S_{2,q',\curl}^{p/(p-q')}>0.
\]
For \(s>0\) and \(y\in\r3\), let
\[
(\tau_{s,y}u)(x):=s^{3/p}u(sx+y).
\]
It is easy to see that \(\tau_{s,y}\) preserves the \(L^p\)-norm of $u$, and using the relation \(3/p+2=3/q'\), also the \(L^{q'}\)-norm of
	\(\curlop\curlop u\) (so $\tau_{s,y}$ is an isometry). It also leaves the spaces \(\cV\) and \(\cW\) invariant.
	The uniqueness of \(w(v)\) therefore yields
	\[
	m(\tau_{s,y}v)=\tau_{s,y}m(v).
	\]
	Since \(|\widehat v_n|_p\geq|\widehat u_n|_p\), $\wh v_n\not\to 0$ in $L^p(\r3,\r3)$.
Hence, passing to a subsequence and applying Solimini's theorem \cite[Theorem~1]{Solimini} 
(cf. \cite[Lemma 3.2]{MedSz2}) we find \(s_n>0\) and \(y_n\in\r3\) such that $\tau_{s_n,y_n}\wh v_n\rh v\ne 0$. Thus $\tau_{s_n,y_n}\wh u_n\rh u\ne 0$. Observe that the result in \cite{Solimini} is for first order Sobolev spaces, therefore we needed boundedness of $\wh v_n$ in $\cD^{1,\frac{3q'}{3-q'}}(\r3,\r3)$ and the identity $p = \frac{3q'}{3-q'}$.
Since the problem is invariant with respect to the action of $\tau_{s,y}$, we may assume without loss of generality that $s_n=1$, $y_n=0$. For notational convenience we relabel \(\wh u_n\) and \(\wh v_n\) as
\(u_n\) and \(v_n\). This is a bounded Palais--Smale sequence
and \eqref{eq:PS-masses}, \eqref{eq:ground-level} are valid for $u_n$. 
Passing to a subsequence again and using Lemma
\ref{Lem:WConcentration} we obtain
\[
u_n=v_n+w(v_n)\rightharpoonup u:=v+w(v),
\qquad
u_n\to u\quad\text{a.e. in }\r3.
\]
After passing to a further
subsequence, all assumptions of Proposition
\ref{conccomp} are satisfied.

It remains to show that \(u\) is a critical point of \(J\) and a minimizer for \eqref{eq:S-reduced}. Recall
$z_n=v_n-v$
and let \(\Sigma\) be the finite concentration set given in
Proposition \ref{conccomp}. Since
\[
|z_n|^p\rh
\sum_{x_j\in\Sigma}\rho_j\delta_{x_j} \qquad \text{in } \cM(\r3)
\]
and the set $\Sigma$ is finite, it is easy to see that
\begin{equation}\label{eq:z-strong-away}
z_n\to0
\quad\text{in }L^p(K')
\quad\text{for every }K'\Subset\r3\setminus\Sigma.
\end{equation}
Choose a compact set  $K\Subset\r3\setminus\Sigma$ and take
\(\eta\in\cC_0^\infty(\r3,[0,1])\) with
\(\eta=1\) on \(K\) and
\(\operatorname{supp}\eta\cap\Sigma=~\varnothing\).
Then \eqref{eq:etaxi2} and \eqref{eq:z-strong-away} applied on $\operatorname{supp}\eta$ give
\[
\int_{\r3}\eta
\big\langle
|\curlop\curlop u_n|^{q'-2}\curlop\curlop u_n,
\curlop\curlop z_n
\big\rangle\,dx
=
\int_{\r3}\eta
\langle |u_n|^{p-2}u_n,z_n\rangle\,dx+o(1)
=o(1).
\]
Since
\(\curlop\curlop z_n=\curlop\curlop(u_n-u)\rightharpoonup0\)
in \(L^{q'}(\R^3,\R^3)\),
\[
\int_{\r3}\eta
\big\langle
|\curlop\curlop u|^{q'-2}\curlop\curlop u,
\curlop\curlop z_n
\big\rangle\,dx=o(1).
\]
Consequently,
\begin{equation} \label{tozero}
\ir3 \eta\big\langle
|\curlop\curlop u_n|^{q'-2}\curlop\curlop u_n - |\curlop\curlop u|^{q'-2}\curlop\curlop u, \curlop\curlop(u_n-u)\big\rangle dx
\end{equation}
tends to 0.
Since \(1<q'<2\), 
\[
\big\langle |a|^{q'-2}a-|b|^{q'-2}b,a-b\big\rangle
\geq c_{q'}\frac{|a-b|^2}{(|a|+|b|)^{2-q'}},
\]
see \cite{Simon}. 
The quotient is defined as zero when $a=b=0$.
Applying this inequality with
\(a=\curlop\curlop u_n\) and \(b=\curlop\curlop u\)
and using \eqref{tozero}, we obtain
\begin{equation} \label{q'<2}
\int_K
\frac{|\curlop\curlop u_n-\curlop\curlop u|^2}
{\big(|\curlop\curlop u_n|+|\curlop\curlop u|\big)^{2-q'}}
\,dx
\longrightarrow 0.
\end{equation}
Moreover, by H\"older's inequality,
\[
\begin{aligned}
\int_K |\curlop\curlop u_n-\curlop\curlop u|^{q'}\,dx
&\leq
\left(
\int_K
\frac{|\curlop\curlop u_n-\curlop\curlop u|^2}
{\big(|\curlop\curlop u_n|+|\curlop\curlop u|\big)^{2-q'}}
\,dx
\right)^{q'/2}
\\
&\quad\times
\left(
\int_K
\big(|\curlop\curlop u_n|+|\curlop\curlop u|\big)^{q'}\,dx
\right)^{(2-q')/2}.
\end{aligned}
\]
The second factor on the right-hand side is bounded since
\((\curlop\curlop u_n)\) is bounded in
\(L^{q'}(\r3,\r3)\). Hence
\[
\curlop\curlop u_n\to\curlop\curlop u
\quad\text{in }L^{q'}(K).
\]
As \(\Sigma\) is finite, using the diagonal argument on 
\[
\bigcup_{m=1}^\infty K_m \quad \text{where} \quad K_m := \ol{B(0,m)} \setminus \bigcup_{j\in I}B(x_j,1/m)
\]
 and passing to  subsequences,
we obtain
\begin{equation} \label{aeconv}
\curlop\curlop u_n\to\curlop\curlop u
\quad\text{a.e. in }\r3.
\end{equation}
Almost everywhere convergence and boundedness imply
$$
|\curlop\curlop u_n|^{q'-2}\curlop\curlop u_n
\rightharpoonup
|\curlop\curlop u|^{q'-2}\curlop\curlop u
\quad\text{in }L^{q}(\R^3,\R^3)$$
and
$$|u_n|^{p-2}u_n\rightharpoonup |u|^{p-2}u
\quad\text{in }L^{p'}(\R^3,\R^3)$$
after taking a further subsequence.
Passing to the limit in \(J'(u_n)\to0\), we conclude that
\[
J'(u)=0.
\]
Now observe that testing \(J'(u)=0\) with \(u\)
	gives
	\[
	\int_{\r3}|\curlop\curlop u|^{q'}\,dx
	=\int_{\r3}|u|^p\,dx.
	\]
	Moreover, \(u=m(v)\), and hence \eqref{eq:S-reduced} implies
	\[
	S_{2,q',\curl}
	\leq
	\left(\int_{\r3}|u|^p\,dx\right)^{1-q'/p}.
	\]
	Consequently,
	\[
	\int_{\r3}|u|^p\,dx
	\geq S_{2,q',\curl}^{p/(p-q')}.
	\]
	On the other hand,  almost everywhere convergence, Fatou's lemma,
	and \eqref{eq:PS-masses} give
	\[
	\int_{\r3}|u|^p\,dx
	\leq\liminf_{n\to\infty}\int_{\r3}|u_n|^p\,dx
	=S_{2,q',\curl}^{p/(p-q')}.
	\]
	Thus
	\[
	\int_{\r3}|\curlop\curlop u|^{q'}\,dx
	=\int_{\r3}|u|^p\,dx=S_{2,q',\curl}^{p/(p-q')}.
	\]
	Hence we infer that  \(u\)
is a minimizer.
Finally, let $\widetilde u\neq0$ be any critical point of $J$.
Then $\widetilde u=m(\widetilde v)$ for some
$\widetilde v\in\cV\setminus\{0\}$ and
\[
\int_{\r3}|\curlop\curlop\widetilde u|^{q'}\,dx
=
\int_{\r3}|\widetilde u|^p\,dx.
\]
Hence, by \eqref{eq:S-reduced},
\[
\int_{\r3}|\widetilde u|^p\,dx
\ge
S_{2,q',\curl}^{p/(p-q')},
\]
and therefore
\[
J(\widetilde u)
\ge
\left(\frac1{q'}-\frac1p\right)
S_{2,q',\curl}^{p/(p-q')}
=
J(u).
\]
Thus $u$ is a ground state.
\end{altproof}

\begin{proof}[Completion of the proof of Theorem \ref{thm:comparison}]
We have already shown that $S_{2,q',\curl} \ge S_{2,q'}$. It remains to prove that this inequality is
strict. Let \(u=m(v)\) be the optimizer obtained above. Suppose to the
contrary that
\[
S_{2,q',\curl}=S_{2,q'}.
\]
Since 
\(|m(v)|_p\leq|v|_p\), we have
\[
S_{2,q'}
=\frac{|\Delta v|_{q'}^{q'}}{|m(v)|_p^{q'}}
\geq\frac{|\Delta v|_{q'}^{q'}}{|v|_p^{q'}}
\geq S_{2,q'}.
\]
Thus both inequalities are equalities and \(v\neq0\) is an optimizer
of the usual vector-valued second-order Sobolev inequality.

Recall from the first part of the proof that $f = |\Delta v|$, $I_2(f)\in \cD^{2,q'}(\r3)$, $|v|\le I_2(f)$ a.e. and $v=I_2(-\Delta v)$. Let $g := -\Delta v$. Then
\[
|v| = |I_2(g)| \leq I_2(f)\quad\text{a.e. in }\r3
\]
which yields
\[
|v|_p \leq \big|I_2(f)\big|_p \leq S_{2,q'}^{-1/q'}|f|_{q'}=|v|_p
\]
(cf. \eqref{sobolev3}).
Therefore
\[
 |I_2(g)(x)|=I_2(f)(x)
\quad\text{ a.e. in }\r3.
\]
Since $v\ne0$, also $f\ne0$, and hence $I_2(f)(x)\ge 0$ a.e. and $\not\equiv 0$. Choose a point $x_0$ such that $I_2(f)(x_0)$ is positive and finite and let
\[
e_0:=\frac{I_2(g)(x_0)}{|I_2(g)(x_0)|}.
\]
Note that once $x_0$ has been chosen, $e_0$ is fixed. Taking inner product with $I_2(g)(x_0)$ gives
\[
\begin{aligned}
0
&= |I_2(g)(x_0)|-e_0\cdot I_2(g)(x_0) =  I_2(f)(x_0)-e_0\cdot I_2(g)(x_0)\\
&=\frac1{4\pi}\int_{\r3}
\frac{|g(y)|-e_0\cdot g(y)}{|x_0-y|}\,dy.
\end{aligned}
\]
As the integrand is nonnegative, we obtain
\[
g(y)=|g(y)|e_0\quad\text{for a.e. }y\in\r3,
\]
and consequently,
\[
v=I_2(g) = I_2(f)e_0.
\]
After rotating the coordinate system we may assume $e_0=(1,0,0)$.  
Since \(v\in\cV\), we have
\[
0=\div v= \frac{\partial}{\partial x_1}\,I_2(f).
\]
As \(I_2(f)\in L^p(\r3)\), this implies $v=0$, a contradiction. Thus
\[
S_{2,q',\curl}>S_{2,q'}.
\]
\end{proof}

\begin{altproof}{Theorem \ref{p=q}}
	We have $p=q=6$. Multiplying the first equation in \eqref{system} by $u$ and the second by $v$ and integrating gives
	\[
	\ir3|\curlop u|^2\,dx = \ir3\langle|v|^4v,u\rangle\,dx \le |v|_6^5|u|_6
	\]
	and
	\[
	\ir3|\curlop v|^2\,dx = \ir3\langle|u|^4u,v\rangle\,dx \le |u|_6^5|v|_6
	\]
	while multiplying the first equation by $v$ and the second by $u$ leads to
	\[
	\ir3\langle\curlop u,\curlop v\rangle\,dx = |u|_6^6 = |v|_6^6.
	\]
	Hence
	\[
	\ir3|\curlop u|^2\,dx + \ir3|\curlop v|^2\,dx \le 2|u|_6^6 = 2|v|_6^6 = 2\ir3\langle\curlop u,\curlop v\rangle\,dx,
	\]
so $\ir3(|\curlop u-\curlop v|)^2\,dx \le 0$ and $\curlop u = \curlop v$ a.e.
 But then, multiplying the first and the second equation in \eqref{system} by $\vp\in\cC_0^\infty(\r3,\r3)$ gives
	\[
	\ir3|v|^4\langle v,\vp\rangle\,dx = \ir3\langle\curlop u,\curlop\vp\rangle\,dx = \ir3\langle\curlop v,\curlop\vp\rangle\,dx =  \ir3|u|^4\langle u,\vp\rangle\,dx.
	\]
	Since $\vp$ is arbitrary, $|u|^4u=|v|^4v$ and thus $u=v$ a.e.
\end{altproof}

\appendix

\section{Concentration-compactness for the classical Lane--Emden system}
\label{app:scalar-compactness}

The concentration argument used above also applies to the scalar $q'$-bilaplacian.
We give a local statement which requires neither symmetry nor a prescribed
Palais--Smale level. 

Throughout the appendix, let $N\geq3$ and
\[
 p,q>\frac{N}{N-2},\qquad \frac1p+\frac1q=\frac{N-2}{N}.
\]
As before, $q'=q/(q-1)$, hence $1<q'<N/2$ and
$p=Nq'/(N-2q')$. The space $\cD^{2,q'}(\rn)$ is endowed with the norm $|\Delta u|_{q'}$ and the functional $J$ is now given by
\[
J(u):=\frac1{q'}\int_{\rn}|\Delta u|^{q'}\,dx -\frac1p\int_{\rn}|u|^p\,dx.
\]

\begin{Prop}\label{prop:scalar}
Let 
$u_n\rightharpoonup u$ in $\cD^{2,q'}(\rn)$.
Assume that
\begin{equation}\label{eq:scalar2}
 J'(u_n)\vp\longrightarrow0,
 \qquad J'(u_n)[\psi(u_n-u)]\longrightarrow0
\end{equation}
for every $\vp,\psi\in \cC_0^\infty(\rn)$.
After passing to a subsequence, there is a finite set
$\Sigma\subset\rn$ such that
\[
 u_n\longrightarrow u\quad\hbox{in }W^{2,q'}_{\mathrm{loc}}(\rn\setminus\Sigma)
 \qquad \text{and} \qquad
 \Delta u_n\longrightarrow\Delta u \quad\hbox{a.e. in } \rn.
\]
In particular, $J'(u)=0$.
\end{Prop}

The above result corresponds to Proposition A.1 and Lemma A.3 in \cite{cs}. As mentioned in the introduction, it also gives an additional information on local $W^{2,q'}$-convergence. The arguments in the proof below are similar but simpler than those in Section \ref{sec:cca}. Therefore we omit some details.

\begin{proof}
Put $z_n:=u_n-u$.
Below we will pass to subsequences whenever necessary, without explicitly mentioning it.
Local elliptic estimates and the Rellich--Kondrachov
theorem give
\begin{equation}\label{eq:scalar3}
 z_n\longrightarrow0\quad\hbox{in }W^{1,q'}_{\mathrm{loc}}(\rn),
 \qquad z_n\longrightarrow0\quad\hbox{a.e. in }\rn.
\end{equation}
Let $\mu$ and $\rho$ be positive measures such that 
\[
|\Delta z_n|^{q'}\rh \mu \quad \text{and} \quad |z_n|^p\rh \rho \quad \text{in } \cM(\rn).
\]
For each $\vp\in \cC_0^\infty(\rn)$,
\[
 \Delta(\vp z_n)=\vp\Delta z_n+
       2\nabla\vp\cdot\nabla z_n+(\Delta\vp)z_n,
\]
and by
\eqref{eq:scalar3}, the last two terms tend to zero in $L^{q'}(\rn)$. 
By \cite[Concentration--Compactness Lemma II]{Struwe}, we have
\begin{equation}\label{eq:atoms}
 \rho=\sum_{j\in I}\rho_j\delta_{x_j},\quad \mu \ge \sum_{j\in I}\mu_j\delta_{x_j} \quad \text{and} \quad
 \mu_j\geq S_{2,q'}\rho_j^{q'/p},\quad \rho_j>0,
\end{equation}
where $I$ is at most countable.

Next we derive inequalities between $\mu_j $ and $\rho_j$ which correspond to those  appearing at the end of the proof of Proposition \ref{conccomp}.
Recall that $A_s(a) := |a|^{s-2}a$. For every nonnegative $\vp\in \cC_0^\infty(\rn)$, the second
assumption in \eqref{eq:scalar2} and an adaptation of \eqref{eq:etaxi}-\eqref{eq:etaxi2} give
\begin{equation}\label{eq:scalar}
 \irn\vp A_{q'}(\Delta u_n)\Delta z_n\,dx
     =\irn\vp A_p(u_n)z_n\,dx+o(1).
\end{equation}
Note that here we have used the fact that $\Delta(\vp z_n) = \vp\Delta z_n + o(1)$.
As in \eqref{elementary},
\[
 |a|^{q'}\lesssim A_{q'}(a+b)a+|b|^{q'}\qquad(a,b\in\R).
\]
Applying this inequality with $a=\Delta z_n$, $b=\Delta u$ and using \eqref{eq:scalar} gives
\[
\irn\vp |\Delta z_n|^{q'}\,dx
\lesssim
\irn \vp|u_n|^{p-1}|z_n|\,dx+\irn \vp|\Delta u|^{q'}\,dx+o(1).
\]
Since \(u_n=z_n+u\), in the same way as in the proof of Proposition \ref{conccomp}, by the triangle and the H\"older inequalities we obtain
\[
\irn\vp|u_n|^{p-1}|z_n|\,dx \lesssim \irn\vp|z_n|^p\,dx + \left(\irn\vp|u|^p\,dx\right)^{1/p'}\left(\irn\vp|z_n|^p\,dx\right)^{1/p}.
\]
Letting $\vp_{j,\eps}$ be as in \eqref{phije}, using the preceding two inequalities and letting $n\to\infty$ gives 
\[
\irn\vp_{j,\eps}\,d\mu \lesssim \irn\vp_{j,\eps}\,d\rho + \left(\irn\vp_{j,\eps}|u|^p\,dx\right)^{1/p'}\left(\irn\vp_{j,\eps}\,d\rho\right)^{1/p} + \irn \vp_{j,\eps}|\Delta u|^{q'}\,dx,
\]
and thus, taking $\eps\to 0$, we obtain $\mu_j\le D\rho_j$. So by \eqref{eq:atoms},
\[
 S_{2,q'}\rho_j^{q'/p}\leq\mu_j\leq D\rho_j
\]
and
\begin{equation}\label{eq:scalarrho}
 \rho_j\geq\left(\frac{S_{2,q'}}{D}\right)^{p/(p-q')}.
\end{equation}
Since $\rho(\rn)<\infty$, the set
$\Sigma:=\{x_j:j\in I\}$ is finite.

Let $K\Subset\rn\setminus\Sigma$ and choose
$\eta\in \cC_0^\infty(\rn,[0,1])$ such that $\eta = 1$ on $K$ and $\supp\,\eta\cap\Sigma=\emptyset$. By \eqref{eq:atoms}, $z_n\to0$ in
$L^p(\operatorname{supp}\eta)$. Hence the right-hand side of
\eqref{eq:scalar} tends to zero. Weak convergence of
$\Delta z_n$ also gives
$\irn\eta A_{q'}(\Delta u)\Delta z_n\,dx\to0$. Subtracting, we find that
\[
 \irn\eta \bigl(A_{q'}(\Delta u_n)-A_{q'}(\Delta u)\bigr)
                     (\Delta u_n-\Delta u)\,dx\longrightarrow0.
\]
If $q'\geq2$, the inequality
$(A_{q'}(a)-A_{q'}(b))(a-b)\geq c_{q'}|a-b|^{q'}$ implies $\int_K|\Delta z_n|^{q'}\,dx\to 0$. If $1<q'<2$, then
\[
 (A_{q'}(a)-A_{q'}(b))(a-b)
       \geq c_{q'}\frac{|a-b|^2}{(|a|+|b|)^{2-q'}},
\]
with the quotient defined as zero when $a=b=0$. Then, as in \eqref{q'<2} and the inequality following it, we obtain 
\[
\int_K|\Delta z_n|^{q'}\,dx \le  \left(\int_K\frac{|\Delta z_n|^2}
          {(|\Delta u_n|+|\Delta u|)^{2-q'}}\,dx\right)^{q'/2}
 \left(\int_K(|\Delta u_n|+|\Delta u|)^{q'}\,dx\right)^{(2-q')/2} \longrightarrow 0.
\]
So in both cases $\Delta z_n\to 0$ and hence $\Delta u_n\to \Delta u$ a.e. in $K$. Now the diagonal argument as in the proof of Theorem  \ref{thm:main} implies a.e. convergence in $\rn$. This also implies that $J'(u)=0$. Finally, since $\Delta z_n\to 0$ in $L^{q'}(K)$ and $z_n\to 0$ in $L^p(K)$ for each $K$ as above, $u_n\to u$ in $W^{2,q'}_{\mathrm{loc}}(\rn\setminus\Sigma)$.
\end{proof}

\begin{Rem}\label{rem:cs}
\emph{
	Proposition \ref{prop:scalar} also applies to the equivariant setting of \cite{cs}. Let $G$ be a
	closed subgroup of $O(N)$, $\phi:G\to\{-1,1\}$ a continuous homomorphism and let $u_n,u$ be
	$\phi$-equivariant, i.e. $u_n(gx)=\phi(g)u_n(x)$. If \eqref{eq:scalar2} holds for all
	$\phi$-equivariant $\vp\in\cC_0^\infty(\rn)$ and all $G$-invariant $\psi\in\cC_0^\infty(\rn)$
	(these are the hypotheses of \cite[Proposition A.1]{cs}), then it holds for all
	$\vp,\psi\in\cC_0^\infty(\rn)$. }
	
	\emph{
	We also point out a gap in the proof of \cite[Proposition A.1]{cs}. There the a.e.
	convergence $\Delta u_n\to\Delta u$ follows from \cite[Lemma A.3]{cs} whose hypothesis
	\begin{equation*}
	\int_{\rn}\psi\,|\Delta u_n|^{q'-2}\Delta u_n\,\Delta(u_n-u)\,dx\to0
	\qquad\text{for all $G$-invariant }\psi\in\cC_0^\infty(\rn)
	\end{equation*}
	is derived by approximating $u_n-u$ in $L^p_{loc}(\rn)$ by mollifications uniformly in $n$.
	This is not possible because $(u_n-u)$ is not relatively compact in $L^p_{loc}(\rn)$. Proposition \ref{prop:scalar} fills this gap, and we expect it
	to be useful in further studies of Lane--Emden systems on the critical hyperbola.
	}
\end{Rem}

\section*{Statements and Declarations}
\noindent{\textbf{Funding.} I. Arkoudis and J. Mederski were partly supported by the National Science Centre, Poland (Grant No. 2023/51/B/ST1/00968). A. Szulkin was partly supported by a grant from the Royal Swedish Academy of Sciences.}

\medskip
\noindent{\textbf{Competing interests.}}
The authors have no relevant financial or non-financial interests
to disclose.

\medskip
\noindent{\textbf{Data availability.}}
	No datasets were generated or analyzed in this theoretical study.

\end{document}